\documentclass{amsart}

\usepackage{amsmath,tabu}
\usepackage{amsfonts}
\usepackage{amssymb}
\usepackage{amsthm}
\usepackage[bookmarks]{hyperref}
\hypersetup{colorlinks=false}
\usepackage{cleveref}
\usepackage{setspace}
\usepackage{enumerate}
\usepackage[leqno]{amsmath}

\makeatletter
\newcommand{\leqnomode}{\tagsleft@true}
\newcommand{\reqnomode}{\tagsleft@false}
\makeatother

\newcommand{\N}{\mathbb{N}}

\newcommand{\lmin}{\mathrm{lmin}}

\newcommand{\lmax}{\mathrm{lmax}}

\theoremstyle{thmit}
\newtheorem{theorem}{Theorem}[section]
\newtheorem{definition}[theorem]{Definition}
\newtheorem{proposition}[theorem]{Proposition}
\newtheorem{lemma}[theorem]{Lemma}
\newtheorem{corollary}[theorem]{Corollary}

\theoremstyle{definition}
\newtheorem{example}[theorem]{Example}
\newtheorem{remark}[theorem]{Remark}

\title{Local operator system structures and their tensor products}
\date{\today}
\author[S. Beniwal]{Surbhi Beniwal}
\address{Department of Mathematics\\ University of Delhi\\ Delhi-110007, INDIA}
\email{surbhinrw@gmail.com }

\author[A.Kumar]{Ajay Kumar}
\address{Department of Mathematics\\ University of Delhi\\ Delhi-110007, INDIA}
\email{ak7028581@gmail.com}

\author[P. Luthra]{Preeti Luthra\textsuperscript{*}}
\address{Department of Mathematics, Mata Sundri College for Women\\ University of Delhi\\ New Delhi-110002, INDIA}
\email{maths.preeti@gmail.com}

  
  \thanks{ \textsuperscript{*}Corresponding author}

\begin{document}
 

\begin{abstract}
We introduce and explore the theory of tensor products in the category of local operator systems. Analogous to minimal operator system OMIN and maximal operator system OMAX, minimal and maximal local operator system structures LOMIN and LOMAX, respectively, are also discussed. 

\end{abstract}
\maketitle
\section{Introduction}

The study of locally $C^*$-algebra was initiated by Inoue in \cite{inoue1972locally}. It was introduced as a complete locally m-convex $*$-algebra with $C^*$-condition where topology is defined by a family of $C^*$-norms and was proved to be isomorphic to a closed $*$-subalgebra of the operator algebra in a generalized Hilbert space. Arveson called these algebra as Pro $C^*$-algebra (see \cite{frago2005top}), which can be represented as projective limit of $C^*$-algebras. Their morphisms and tensor products in this category have been extensively studied (see \cite[Section 31]{frago2005top}). Further this theory was extended to operator spaces by Webster in \cite{webster1997local}, and their tensor products were also introduced and explored \cite{4,5,webster1997local}. Dosiev in \cite{dosiev2008local} revisited local operator spaces, introduced local operator systems in a concrete way and gave a locally convex version the Choi-Effros representation theorem. In \cite{anardosi}, Dosiev used quantum cones with filtered base to classify local operator systems, which he called as quantum systems. Again local operator systems can be represented as projective limits of operator systems. Now, in the last decade the Paulsen et. al. \cite{kavruk2011tensor} introduced the abstract definition of operator systems and initiated a systematic study of their tensor products. This research contributed a lot in the field of quantum information theory. Operator systems are considered as quantum versions of function systems.

In this paper, we give a short proof of the fact that the abstract definition of local operator systems is equivalent to the definition of local operator systems. In the same section, we show that local operator systems possess the structure of local operator spaces. We also prove that these abstract local operator systems are in fact the projective limit of abstract operator systems.

In \cite{8}, the authors associated with any Archimedean order unit space $V,$ two
operator systems, the minimal OMIN(V) and the maximal OMAX(V). In Section \ref{s:3}, we consider a parallel development for local operator systems. 

Motivated by the theory of tensor products of operator systems, in Section \ref{s:4} we give the definition of tensor product in the category of local operator systems. We introduce minimal and maximal local operator system tensor products and explore their properties. 

\section{Preliminaries}\label{s:1}
In paper \cite{dosiev2008local}, it was proved that every local operator space can be realized as a linear space of unbounded operators on a Hilbert space and gave the following theory: For a fixed Hilbert space $H$, an upward filtered family of closed subspaces $\mathcal{E} = \{H_{\alpha}\}_{\alpha \in \Lambda}$ such that their union $\mathcal{D}$ is a dense subspace in $H$ with $p =  \{P_{\alpha}\}_{\alpha \in \Lambda}$ family of projections in $B(H)$ onto the subspaces $H_{\alpha},\alpha \in \Lambda$. The algebra $C_\mathcal{E} (\mathcal{D})$ of all non-commutative continuous functions on a quantized domain $E$ is given by $$C_\mathcal{E} (\mathcal{D}) = \{T \in L(\mathcal{D}): TP_{\alpha}= P_{\alpha}TP_{\alpha} \in B(H), \alpha \in \Lambda\},$$
where $L(\mathcal{D})$ is the associative algebra of all linear transformations on $\mathcal{D}$ and $\alpha \leq \beta$ whenever $P_{\alpha} \leq P_{\beta}$, so that $H_{\alpha} \subseteq H_{\beta}$. Thus each $T \in C_\mathcal{E} (\mathcal{D})$ is an unbounded operator on $H$ with domain $\mathcal{D}$ such that $T (H_{\alpha}) \subseteq H_{\alpha}$ and $T|_{H_{\alpha}} \in B(H_{\alpha})$, and $C_\mathcal{E} (\mathcal{D})$ is a subalgebra in $L(\mathcal{D})$. The set
$$C^*_\mathcal{E} (\mathcal{D}) =\{T \in C_\mathcal{E} (\mathcal{D}) \; : \; P_{\alpha}T \subseteq TP_{\alpha}, \alpha \in \Lambda\}$$
of all non-commutative continuous functions on $E$ is a unital $*$-subalgebra of $C_\mathcal{E}(\mathcal{D})$, with the involution $T^* = T^{\bigstar}|_D \in  C^*_\mathcal{E} (\mathcal{D})$ for all $T \in C^*_\mathcal{E}  (\mathcal{D})$ where $T^{\bigstar}$ is unbounded dual of $T$ such that $\mathcal{D} \subseteq \mathrm{dom}(T^{\bigstar})$ and $T^{\bigstar}(\mathcal{D}) \subseteq \mathcal{D}$.

A \emph{concrete local operator space}  is defined as a subspace of $C_{\mathcal{E}}(\mathcal{D})$, whereas a \emph{concrete local operator system} $\mathcal{S}$ as a unital self-adjoint subspace of the multinormed $C^*$-algebra $C^*_\mathcal{E} (\mathcal{D})$. 

Any linear space $\mathcal{E}$ with a separated family of matrix seminorms $\{p_{\alpha}:\alpha \in \Gamma\}$ is called an \emph{abstract local operator space}. These two definitions, abstract and concrete, of local operator spaces were proved equivalent in \cite[Theorem 7.1]{dosiev2008local}. 

Let us establish some more terminologies related to local positivity that shall be used throughout:
\begin{enumerate}
\item An
element $T$ of a local operator system $\mathcal{S}$ is \emph{local hermitian} if $T = T^*$ on a certain subspace $H_{\alpha}$,
that is, $T |_{H_{\alpha}} = T^*|_{H_{\alpha}} = (T |_{H_{\alpha}})^* \in B(H_{\alpha})$. Respectively, an element $T \in \mathcal{S}$ is said to be \emph{local
positive} if $T |_{H_{\alpha}} \geq 0$ in $B(H_{\alpha})$ for some $\alpha \in \Lambda$.
\item  Further, let $\mathcal{S} \subseteq C^*_\mathcal{E} (D)$ and $\mathcal{T} \subseteq C^*_\mathcal{F} (O)$ be
local operator systems on quantized domains $\mathcal{E} = \{H_{\alpha}\}_{\alpha \in \Lambda}$ and $F = \{K_{i}\}_{i \in \omega}$ with their union
spaces $D$ and $O$, respectively. A linear mapping $\phi :\mathcal{S} \rightarrow \mathcal{T}$ is said to be local matrix positive
if for each $i \in \omega$ there corresponds $\alpha \in \Lambda$ such that $\phi^{(n)}(v)|_{K^n_i} \geq 0$ whenever $v|_{H^n_{\alpha}} \geq 0$, and
 $\phi^{(n)}(v)|_{K^n_i} = 0$ whenever $v|_{H^n_{\alpha}} = 0$, $v \in M_n(\mathcal{S})$, $n \in \mathbb{ N}$, where $\phi^{(n)} :M_n(\mathcal{S} )\rightarrow M_n(\mathcal{T})$, $n \in \mathbb{ N}$, are the
canonical linear extensions of $\phi$ over all matrix spaces.
\end{enumerate}

As we are dealing with projective limit of operator sytem so we refer the reader to \cite{paulsen,2009vector,8,kavruk2011tensor} for a detailed theory on operator systems, their structures and tensor products.

\section{Abstract local operator systems}\label{s:2}
We start by defining a local version of ordered $*$-vector space, which can also be found in \cite{asadi}.

           \begin{definition}
           Let $V$ be a $\ast$-vector space consisting of downward filtered family of cones $\{\mathcal{C}_\alpha$ : $\alpha \in \Gamma\}$ satisfying two properties
            $\mathcal{C}_\alpha$ is a cone (need not be proper) in $V$ where $\mathcal{C}_\alpha \subseteq V_h=\{v\in V : v^{*}=v\}$.
            and $\displaystyle \bigcap_\alpha(\mathcal{C}_\alpha \cap -\mathcal{C}_\alpha)=\{0\}.$ 
           
 Then $V$ is called \emph{local $*$-ordered vector space} and the elements of $\mathcal{C}_\alpha$ are called local positive elements, denoted by $v\geq_\alpha0$. Also, we write $v_1\geq_\alpha v_2$ if $v_1-v_2\geq_\alpha 0$ in $V$.

          \end{definition} 
    \begin{definition}
           For a local $*$-ordered vector space $(V,\{\mathcal{C}_\alpha : \alpha \in \Gamma\})$, an element $e\in V_{h}$ is called \emph{an ordered unit for $V$ }if for all $v\in V_{h}$ and for every $\alpha \in \Gamma$ there exists $r_\alpha\textgreater 0$ such that $r_\alpha e$ $\geq_\alpha$ $v$. If, in addition, $e$ is an Archimedean order unit, we call the triple $(V,\{\mathcal{C}_\alpha \;:\;\alpha \in \Gamma\},e)$ \emph{an Archimedean local $\ast$-ordered vector space} or in short \emph{A.L.O.U space}.
        \end{definition}
        
If $V$ is a $*$-vector space then $M_n(V)$ is also a $*$-vector space. For $A=[v_{ij}]_{n \times n} \in M_n(V)$ we have $A^{\ast}=[v^{\ast} _{ji}]$ is an involution on $M_n(V)$ and $M_n(V)_{h}=\{A \in M_n(V) : A^{\ast}=A\}$.

           \begin{definition}
           Let $V$ be a $*$-vector space. We say that the family $\left\{\{\mathcal{C}_\alpha ^n\}_{n=1} ^{\infty} \; : \; \alpha \in \Gamma\right\}$ is a \emph{local matrix ordering on $V$} if 
           \begin{enumerate}
           \item $(M_n(V),\{\mathcal{C}_\alpha^n: \alpha \in \Gamma\})$ is a local $*$-ordered vector space for each $n \in \N,$ 
           \item for each $n,m\in \mathbb{N}$ and $X \in M_{n,m}$ and all $\alpha$, we have that $X^*\mathcal{C}_\alpha ^n X \subseteq \mathcal{C}_\alpha ^m$. 
           \end{enumerate}
           In this case, we call $\left(V, \left\{\{\mathcal{C}_\alpha ^n\}_{n=1} ^{\infty}\; :\; \alpha \in \Gamma\right\}\right)$ a local matrix $*$-ordered vector space.
        \end{definition}

            For $e\in V_h,$ let $e_n= \mathrm{diag}(e,e,\ldots,e)$ be the corresponding diagonal matrix in $M_n(V)$. We say that $e$ is \emph{a matrix order unit for $V$} if $e_n$ is an order unit for $\left(M_n(V),\left\{\{\mathcal{C}_\alpha ^n \; : \; \alpha \in \Gamma\}\right\}\right)$ for each $n$. We say that $e$ is an \emph{Archimedean matrix order unit} if $e_n$ is an Archimedean order unit for $\left(M_n(V),\left\{\{\mathcal{C}_\alpha ^n \; :\; \alpha \in \Gamma\}\right\}\right)$ for each $n$. And call, $\left(V, \left\{\{\mathcal{C}_\alpha ^n\}_{n=1} ^{\infty}; \alpha \in \Gamma\right\},e\right)$ \emph{an Archimedean local matrix ordered vector space}.

           	\begin{definition}\label{localopsys}
           	An \emph{abstract local operator system} is a triple $\left(V, \left\{\{\mathcal{C}_\alpha ^n\}_{n=1} ^{\infty}; \alpha \in \Gamma\right\},e\right)$, where $V$ is a local $*$-ordered vector space, $\left\{\{\mathcal{C}_\alpha ^n\}_{n=1} ^{\infty}\; : \;\alpha \in \Gamma\}\right\}$ is a local matrix ordering on $V$ and $e$ is an Archimedean matrix order unit. 
           	\end{definition}
           	
           	As should be the case, we first observe that local operator systems defined in the sense of \cite{dosiev2008local} are also abstract local operator systems in the sense of definition \ref{localopsys}. 
           	
           	\begin{remark}\label{concisabs} Every concrete local operator system is abstract local operator system.\\
                  	 
                  	 \end{remark}
           
           	\begin{example}\label{3.6} We now list some more examples of local operator systems:
           		\begin{enumerate}[(i)]
           			\item Every operator system is a local operator system.
           	\item Consider the set $C(\mathbb{R})$, the space of all complex valued continuous function S defined on $\mathbb{R}$.
           	For every compact subset $K$ of $\mathbb{R}$, we define a family of cones as $C_{K} =\{f \in  C(\mathbb{R})_{h}$ : $f(x) \geq 0 \forall x\in K \}$ and Archimedean matrix order unit is I($x$)=1 $\forall x \in \mathbb{R}$, then $\left(C(\mathbb{R}),\left\{\{C_{K}^n\}: K\subseteq\mathbb{R}, K \text{is compact} \right\}, I\right)$ is a local operator system.
           	
		\item Let $H$ be Hilbert space, $\mathcal{D}$ is dense subspace in H and $C_{\mathcal{E}}^*(\mathcal{D})$ be as defined in Remark \ref{concisabs}. Take $T \in C_{\mathcal{E}}^*(\mathcal{D})$ we have $\mathcal{LOS}(T)= \mathrm{span} \{I, T, T^*\}$ is a local operator subsystem of $C_{\mathcal{E}}^*(\mathcal{D})$.  
	\item	Let $M_\infty=\{[a_{ij}]: a_{ij}\in \mathbb{C}; i,j \in \mathbb{N}\}$.
		For each $\alpha_n=n\in \mathbb{N}$, we have $C_{\alpha_n}^1=\{[a_{ij}]\in M_\infty: [a_{ij}]\in M_n^+$ for $1\leq i,j \leq n \}$. For higher orders, we have $C_{\alpha_n}^m=\{[A_{ij}]\in M_m(M_\infty): [A_{ij}|_n] \geq 0\}$ where $A_{ij}|_n $ denotes restriction to $n \times n$ coordinates. With Archimedean local matrix order unit $I=[e_{ij}]$, where $e_{ij}=\delta_{ij}$ is Archimedean local matrix order unit, clearly $M_\infty$ is a local operator system.
		\item We have another local operator subsystem of $M_\infty$ which is tridiagonal matrices i.e. $T_\infty=\mathrm{span}\{E_{ij}: |i-j|\leq 1\}$.
		\item 
		 On the similar lines of remark 5.19 in \cite{kavruk2011tensor}, we can define graph $G_\infty$ on $\mathbb{N}$ which can be identified with a subset $G_\infty \subseteq \mathbb{N}\times \mathbb{N}$ satisfying the properties that $(i,j)\in G_\infty$ whenever $(j,i)\in G_\infty$ and $(i,i)\in G$ and $S(G_\infty)=\mathrm{span}\{E_{ij}:(i,j)\in G_\infty\}\subseteq M_\infty$ is a local operator system.
		\item 	We have an important algebra given by Cuntz which is Cuntz algebra. The Cuntz algebra $\mathcal{O}_n$ ($n\geq2$) is the universal C*-algebra generated by n isometries $S_1,S_2,...,S_n$ with $\sum_{i=1}^{n}S_iS_i^*=I$ where $I$ is the identity operator, for details refer to \cite{cuntz}. On the similar lines we tried to define a Pro C*-algebra as follows. Let $\mathcal{D}$ be a quantized domain of separable Hilbert spaces $H_\alpha$ where $\alpha \in \Gamma$. Let $\mathcal{O}^l_n$ denotes the Cuntz Pro C*-algebra generated by n elements $S_1,S_2,..,S_n$ having the properties $S_i^*S_i|_{H_\alpha}=I_{H_\alpha}$ for all $i$ and for all $\alpha$  
			 and $\sum_{i=1}^{n}S_iS_i^*|_{H_\alpha}=I_{H_\alpha}$ for all $\alpha$
		We now define Cuntz local operator system.
	Let $\mathcal{O}^l_n$ be Cuntz Pro $C^*$-algebra generated by n elements $S_1,S_2,..,S_n$ having the above two properties. Then we denote by $\mathcal{SO}^l_n$ the Cuntz local operator system defined as 
		$\mathcal{SO}^l_n$:= span \{$I,S_1,S_2,..,S_n,S_1^*,S_2^*,..,S_n^*\}$.
	\end{enumerate}
           \end{example}
           
   
  In the next result we show that with respect to an induced family of seminorms, local operator systems possess the structure of local operator space.    
   
   \begin{proposition}\label{seminorm} Let $V$ be a local operator system with family of cones $\left\{\{\mathcal{C}_\alpha ^n\}_{n=1} ^{\infty}\; : \; \alpha \in \Gamma\}\right\}$ and $e$ is Archimedean matrix order unit and for each $X \in M_n(V) $ set $\Vert X \Vert_{\alpha}^n= \inf\left\{r\geq 0:\begin{pmatrix}
       re_n & X\\
       X^{*} & re_n
       \end{pmatrix} \in \mathcal{C}_\alpha ^{2n} \right\},$ then $\Vert \cdot \Vert_\alpha^n $ is a separating family of $\ast$-seminorms on $M_n(V)$ and $\mathcal{C}_\alpha ^n$ is a closed subset of $M_n(V)$ in the topology induced by this separating family of $\ast$-seminorms. Hence, $\left\{V, \left\{\|\cdot\|_\alpha^n\right\}_{n=1}^\infty\right\}$ is a local operator space.
    \end{proposition}
       \begin{proof}   By some simple computations one can show that $\|\cdot \|_{\alpha}^n $ is a family of separating $*$-seminorms.    

       Finally, we show that $\mathcal{C}_\alpha ^1$ is closed in the locally convex topology generated by the separating family of seminorms.
       Let $x$ be any limit point of $\mathcal{C}_\alpha^1$ then 
        V($\alpha$,$n$) $\cap$ $C^1_\alpha$ $\neq$ $\phi$ $\forall$ $n\in \mathbb{N}$ and $\forall$ $\alpha$ where V($\alpha$,$n$)=\{$y\in V$:$\|y-x\|_\alpha^1 \textless \frac{1}{n}$\}.
         Note that 
         $x=x^*$.
       Given any $r$ $\textgreater$ 0 choose $n\in \mathbb{N}$ such that $\frac{1}{n}
       \textless r$. 		          	       
       Let $y\in V(\alpha,n) \cap \mathcal{C}_\alpha^1$ $\implies$ $\Vert x-y \Vert_\alpha$ $\textless r$  $\implies$ $\begin{pmatrix}
       re & x-y\\
       x-y & re
       \end{pmatrix}$ $\in \mathcal{C}_\alpha^2$ 
        This yields $2re+2x-2y = (1,1) \begin{pmatrix}
       re & x-y\\
       x-y & re
       \end{pmatrix}  (1,1)^*\in \mathcal{C}_\alpha^1$ $\implies$ $re+xe$ $\in \mathcal{C}_\alpha^1,$ so that using Archimedean property $x\in \mathcal{C}_\alpha^1$. Hence $\mathcal{C}_\alpha^1$ is closed.
       Similarly, one can prove that $\mathcal{C}_\alpha^n$ is closed $\forall$ $n\in \mathbb{N}$.
       \end{proof}

\subsection{The Archimeadeanization of a local $\ast$-ordered vector space}
           	Let ($V,\{\mathcal{C}_\alpha\}_{\alpha \in \Gamma},e$) be a local ordered vector space, where $V$ is a $\ast$- vector space, \{$\mathcal{C}_\alpha$\} family of cones, $e$ is an ordered unit that is not Archimedean. The process of enlarging cones to make the space Archimedean ordered unit space for ordered $*$-vector space was introduced in \cite{2009vector}. We now show that the same process works for local $*$-ordered vector spaces.
           	\begin{definition}
           			Let $\left(V,\left\{\mathcal{C}_\alpha\; : \; \alpha \in \Gamma\}\right\},e\right)$ be a local $\ast$-ordered vector space with an order unit $e$. Define $\mathcal{D}_\alpha$:=\{$v \in V$ : $re+v \in \mathcal{C}_\alpha$ for all $r>0\}$  for each $\alpha \in \Gamma.$
           	\end{definition}
           	
           Clearly $\mathcal{D}_\alpha$ is a cone with $\mathcal{C}_\alpha \subseteq \mathcal{D}_\alpha.$

\begin{proposition}\label{closure}
                     	Let $\left(V,\left\{\{\mathcal{C}_\alpha\}\; : \;\alpha \in \Gamma\right\}\right)$ be a local $\ast$-ordered vector space with an order unit $e$. Then $\mathcal{D}_\alpha$ is equal to the closure of $\mathcal{C}_\alpha$ in topology induced by family of seminorms introduced in Proposition \ref{seminorm}.
\end{proposition}
                 \begin{proof} 
                 	If $v \in \mathcal{D}_\alpha$, then $re+v \in \mathcal{C}_\alpha \forall r>0$
                 	$\implies$ $\|re+v-v\|_\beta$=$r$  \hspace{.2cm}$\forall r>0$\hspace{.2cm} $\forall \beta \in \Gamma$.
                 	So $\mathcal{D}_\alpha\subseteq \bar{\mathcal{C}_\alpha}$. Conversely, for any $v \in \bar{\mathcal{C}_\alpha}$ and any neighbourhood $U$ of $v,$ we have $U \cap \mathcal{C}_\alpha$ $\neq$ $\phi$. Now for $n \in \mathbb{N}$, choose $r_n>0$ such that $r_n$ $\longrightarrow$ 0. Take $U_n$=\{$u\in V$ : $\|u-v\|<r_n$\},  as $U_n$ is a nbd of $v$ so $U_n \cap \mathcal{C}_\alpha$ $\neq$ $\phi$ $\forall n \in \mathbb{N}$.
                 	For each $n \in \mathbb{N}$, assume $u_n \in U_n \cap \mathcal{C}_\alpha$.
                 	Thus $r_n e \pm(u_n-v) \in \mathcal{C}_\alpha$ $\implies$ $r_ne+v-u_n \in \mathcal{C}_\alpha$	$\implies$ $r_ne+v \in \mathcal{C}_\alpha \hspace{.2cm} \forall n$ as $\mathcal{C}_\alpha$ is a cone.
                 	Since $r_n \longrightarrow 0$, it follows that $re+v \in \mathcal{C}_\alpha$ $\forall r>0$.Thus $v \in \mathcal{D}_\alpha$.
                 	Hence we get  $\mathcal{D}_\alpha =\bar{\mathcal{C}_\alpha}$.
                 	\end{proof}

\subsection{Representation theorem for local operator systems}
           	       We now give a short alternative proof of the fact that abstract definition of local operator system given in this paper and concrete definition of local operator system given in \cite{dosiev2008local} are equivalent. 
           	       
           	       \begin{definition}
           	       Let $V$ and $W$ be two abstract local operator system with \{$\mathcal{C}_\alpha$ : $\alpha \in \Gamma\}$ and \{$H_\beta$ : $\beta \in \Omega\}$ family of cones, respectively. A linear map $\Phi$: $V \longrightarrow W$ is called
           	       \begin{itemize}
           	       	\item \emph{unital local positive} if for each $\beta \in \Omega$ there corresponds $\alpha \in \Gamma$ s.t. $\Phi(C_\alpha)\subseteq H_\beta$ and $\Phi(e)=f$ where $e$ and $f$ are the Archimedean local matrix units of $V$ and $W$ resp.
           	       	\item \emph{local order isomorphism} if $\Phi$ is bijective, $\Gamma=\Omega$ and $\Phi(C_\alpha)=H_\alpha$ for all $\alpha.$
           	       \end{itemize} 
           	       \end{definition}
           	        We denote the category of local operator systems with unital local completely positive maps by $\mathcal{LO}$
                 \\ Now we define natural morphism in local operator spaces which is local completely bounded maps as given in \cite[Section 4.3]{dosiev2008local}.
                  
                  \begin{definition}
                  	Let $E=\lim\limits_{\leftarrow}E_\rho$ and $F=\lim\limits_{\leftarrow}F_\sigma$ be two local operator spaces. A linear map $\phi:E \to F$ is said to be locally completely bounded if for every $\sigma \in \Lambda$; there exists a $\rho \in \Gamma$ and K $>$ 0 s.t. $ ||\pi_\sigma^{(n)}\phi^{(n)}(e)||_{M_n(F_\sigma)} \leq K ||\pi_\rho^{(n)}(e)||_{M_n(E_\rho)}$ for every $n \in \mathbb{N}$, $e\in M_n(E_\rho)$ where $\pi_\rho$ and $\pi_\sigma$ are projection maps given by $\pi_\rho:E \to E_\rho$ and $\pi_\sigma: F \to F_\sigma$ resp.
                  	We denote by $||\phi||_{(\rho,\sigma)}^{l.c.b}$ the local complete bound of $\phi$ w.r.t $(\rho,\sigma)$. 
                  \end{definition}

                  	 We have already seen in Remark \ref{concisabs} that every concrete local operator system is abstract. Next result proves that converse also holds true.
           	      	\begin{theorem}[Representation theorem]\label{representation}
           	       	Let $V$ be an abstract local operator system, then there exists a unital complete local order embedding $\Phi$ from $V$ into $C_\mathcal{E} ^*(D)$. Hence abstract local operator systems are equivalent to concrete local operator systems.
           	       	\end{theorem}
           	       	\begin{proof} Let $V$ be a local operator system with family of cones $\{\mathcal{C}_\alpha ^n\}_{\alpha \in \Gamma}$ and Archimedean order unit $e$.
           	       	Put $S_{\alpha}^{(r)}$=U.L.C.P.$_{\mathcal{C}_\alpha}(V,M_r)$, the set of all unital local completely positive maps $w$ : $V\to M_r$ with respect to cone $\mathcal{C}_\alpha$. 
           	       	 Also, put $S_{\alpha}$=$\cup _{r\in \mathbb{N}}$  $S_{\alpha} ^{(r)}$ where $\alpha\in \Gamma.$ It follows that \{$S_{\alpha}$\} is a directed family of sets. Let S=$\displaystyle\cup_{\alpha\in \Gamma}$  $S_{\alpha}$. For any $\alpha \in \Gamma$, consider the Hilbert space $H_{\alpha}$=$\oplus_{w\in S_{\alpha}}\mathbb{C}^{n(w)}$ where $n(w)=n$ whenever $w\in S_{\alpha} ^{(n)}$
           	       		Let us introduce $\Phi_{\alpha} :V \longrightarrow$ B($H_{\alpha}$) defined by  $\Phi_{(\alpha}$($v$)=($w(v))_{w\in S_{\alpha}}.$
                  	We can easily show that $\Phi_{\alpha}$ is unital matrix positive map. 
                  	Define $\Phi : V\longrightarrow C_{\mathcal{E}}(\mathcal{D})$ by 
                  	$\Phi(v)$=$\Phi_{\alpha}(v)$ on $H_{\alpha}$. Clearly, $\Phi$ is unital matrix local order positive map because for each $\alpha \in \Gamma,$ $\Phi(v)\geq _{\alpha} 0$ whenever $v\geq _{\alpha}$ 0. 
                     Now we see that for every $\alpha\in \Gamma$ there exists $\beta\in \Gamma$ s.t. $\Phi(v)\gneq_{\beta}$ 0 whenever $v\gneq _{\alpha}$ 0 using Hahn-Banach Separation Theorem.
                 
                  	  By \cite[Proposition 3.1]{dosiev2008local}, dual of $\Phi(v)$ is $\Phi(v^*)$ thus we have that $\Phi(v)\in  C_{\mathcal{E}}^*(\mathcal{D}),$ i.e. $\Phi(V) \subseteq C_{\mathcal{E}}^*(\mathcal{D})$.
                  Finally it is trivial that $\Phi$ is actually injective.
                   
                  	  \end{proof}
                  	  
                  	 \subsection{Projective limit of operator systems}\label{projective}
                  	  We now consider the projective system of operator systems, and show that its projective limit possesses the structure of an abstract local operator system.

                  	   Let \{($V_\alpha,f_{\alpha \beta}$) $:$ $\alpha \leq \beta$ in $\Gamma$\} be a projective system of operator systems, that is, $\{V_\alpha\; :\; \alpha \in \Gamma\}$ is a family of operator system; $\Gamma$ being a directed set and $\{f_{\alpha\beta}\; : \; V_\beta \to V_\alpha$ $|$ $\alpha \leq \beta$ in $\Gamma$\} a family of unital completely positive maps satisfying $f_{\alpha \beta}$o$f_{\beta \gamma}$=$f_{\alpha \gamma}$ $\forall$ $\alpha \leq \beta \leq \gamma$ in $\Gamma$ and $f_{\alpha \alpha}$=$Id_{V_\alpha}$ $\forall$ $\alpha \in \Gamma$. The set
                  $V=\{(v_\alpha)_{\alpha \in \Gamma} \in \Pi_{\alpha \in \Gamma} V_\alpha \; : \; f_{\alpha\beta}(v_\beta)=v_\alpha \forall \alpha \leq \beta \text{ in } \Gamma\},$ is then the projective limit of projective system of the operator system. $V$ is, in fact, a $\ast$-vector space with involution defined as $(v)^*$=$(v^{*_\alpha})$ and family of cones \{$\mathcal{C}_\alpha$\} where $\mathcal{C}_\alpha$ is defined as $\mathcal{C}_\alpha=\{v=(v_\alpha) \in V:v_\alpha \in V_\alpha^+\}.$ Clearly, $\mathcal{C}_\alpha$ is a cone. Similarly, we can define for higher levels.  Also, here we have that  $v$ $\in \mathcal{C}_\alpha ^n $ and $v$ $\in -\mathcal{C}_\alpha ^n$ $\forall$ $\alpha$ then $v$=0  for each $n\in \mathbb{N}$. Also, $e$=($e_\alpha$), where $e_\alpha$ is the Archimedean order unit of $V_\alpha$. With the above observations, we have that V is a local operator system. We now prove that the converse is also true.
                  
                  \begin{proposition}
                  Every local operator system (abstract) can be obtained by projective limit of operator systems (abstract).
                  \end{proposition}
                  \begin{proof} 
                  We have that V is a local operator system with family of cones \{$\mathcal{C}_\alpha$\}. Let $M_\alpha$=$<\mathcal{C}_\alpha \cap -\mathcal{C}_\alpha>,$ the subspace generated by $\mathcal{C}_\alpha \cap -\mathcal{C}_\alpha$.
                  Let $V_\alpha$:= $V/M_\alpha$, which is a operator system with cone as  $M_n(V_\alpha)^+$=\{[$v_{ij}$+$M_\alpha] \in M_n(V_\alpha)$ : $[v_{ij}] \in \mathcal{C}_\alpha^n$\}.
					 Connecting maps are defined as 
					for all $\alpha \leq \beta \in \Gamma$, define $f_{\alpha\beta}$ : $V_\beta \to V_\alpha$ given by $f_{\alpha\beta}(v+M_\beta)$=$v+M_\alpha,$
					Therefore the above defined family forms a projective system of operator systems.
			\end{proof}

               \begin{remark}
              	We have defined Cuntz local operator system in Examples \ref{3.6}. Now we see that they are the projective limit of Cuntz operator system. Indeed,  $\mathcal{SO}^l_n$=$\lim\limits_{\leftarrow}\mathcal{SO}^\alpha_n$ where $\mathcal{SO}^l_n$ is Cuntz local operator system corresponding to domain $\mathcal{D}=\cup H_\alpha$and $\mathcal{SO}^\alpha_n$ is Cuntz operator system corresponding to $H_\alpha$ Hilbert space with n isometries $S_1|_{H_\alpha},S_2|_{H_\alpha},..,S_n|_{H_\alpha}$.
              \end{remark}
              \begin{remark}
              	We also have an obvious result that Projective limit of Cuntz operator system is Cuntz local operator system. 
              \end{remark}

\section{Local operator system structures on Archimedean local ordered unit spaces}\label{s:3}
              
             \subsection{LOMIN structure on Archimedean local ordered unit space}
              In \cite{asadi}, the minimal local operator system structure of A.L.O.U. spaces has been introduced. We here give a brief overview of this structure and prove some important characterizations. 
              Given an Archimedean local ordered unit space $(V,\{V_\alpha^+$ : $\alpha \in \Gamma\},e)$, define $N_\alpha$=$<V_\alpha^+ \cap -V_\alpha^+>$ and let $M_\alpha$=$V/N_\alpha$. Then $M_\alpha$ is a $\ast$-ordered vector space with cone $M_\alpha^+$=\{$v+N_\alpha \in M_\alpha$ :$v\in V_\alpha^+$\}.
              If $S_\alpha$ denotes the set of all unital positive linear functionals on $M_\alpha$ s.t. $s(M_\alpha^+) \geq 0 \forall s \in S_\alpha$, then by \cite[Theorem 5.2]{2009vector} there exists a compact Hausdorff topology on $S_\alpha(M_\alpha)$ such that $C( S_\alpha(M_\alpha))$ is an operator system. Then $C(\cup S_\alpha(M_\alpha))$ is a local operator system.
              $\ast$-vector space with involution given by $f^*(x)$=$\overline{ f(x)}$ that becomes a local operator system with respect to family of cones given by $\mathcal{D}_\alpha^n=\{(f_{ij}) \in M_n(C(\cup S_\alpha(M_\alpha))) \; : \; (f_{ij})|_{S_\alpha(M_\alpha)} \geq 0\}$ and Archimedean matrix order unit $I$, where $I(x)=1$ $\forall x \in \cup S_\alpha(M_\alpha)$.

              Define $\Phi : V \to C(\cup S_\alpha(M_\alpha))$ by $\Phi(v)|_{S_\alpha(M_\alpha)}(s)=s(v+N_\alpha)$. Clearly, $\Phi$ is a unital local order isomorphism onto its range.
              Also we have that \begin{eqnarray*}
              \Phi^{-1}(\mathcal{D}_\alpha)&= & V_\alpha ^+.\end{eqnarray*}
              Using this map $\Phi$, define local matrix ordering $\{\mathcal{C}_\alpha ^n\}$ on ($V,\{V_\alpha^+\},e$) as:  $(v_{ij}) \in \mathcal{C}_\alpha ^n$ if and only if $\Phi(v_{ij})\geq 0$ in $M_n(C(\cup S_\alpha (M_\alpha)))$. This discussion leads us to following result:

              \begin{proposition}
              	Let $(V,\{\mathcal{C}_\alpha$ : $\alpha \in \Gamma\},e)$ be an Archimedean local ordered unit space. Then $\left(V,\left\{\{(\mathcal{C}_\alpha^n)^{\min}(V)\}_{n=1}^\infty \; : \; \alpha \in \Gamma\right\},e\right)$ is a local operator system induced by the inclusion of $V$ into $C(\cup S_\alpha(V_\alpha)).$
              \end{proposition}

              \begin{definition}
              	Let $(V,\{V_\alpha^+$ : $\alpha \in \Gamma\},e)$ be an Archimedean local order unit space.	For each $n \in \mathbb{N}$ and each $\alpha \in \Gamma$, define
              	$$(\mathcal{C}_\alpha^n)^{\min}(V) : =\{(v_{ij}) \in M_n(V)\; : \; \sum\limits_{i,j=1}^{n}\bar{\lambda_i}\lambda_j v_{ij} \in \mathcal{C}_\alpha \indent \forall\indent \lambda_1,....\lambda_n \in \mathbb{C}.\}$$ We denote $\{(\mathcal{C}_\alpha^n)^{\min}(V)\}_{n=1}^\infty$ for each $\alpha \in \Gamma$ by $\mathcal{C}_\alpha^{\min}(V)$.       	  
              \end{definition}
              
               In the following result, we prove an alternative way to define $\mathcal{C}_\alpha^{\min}(V)$ which implies that with these cone structures $(V,\{\mathcal{C}_\alpha$ : $\alpha \in \Gamma\},e)$ is in fact a minimal local operator system structure possible on $V$.
               
              \begin{theorem}
              	Let $(V,\{\mathcal{C}_\alpha$ : $\alpha \in \Gamma\},e)$ be an Archimedean local ordered unit space. Then ($v_{ij}$) $\in (\mathcal{C}_\alpha^n)^{\min}(V)$ if and only if $(s(v_{ij}+N_\alpha))$ $\in M_n^+$ for each $s\in S_\alpha(M_\alpha)$ for all $\alpha \in \Gamma$ and $n \in \mathbb{N}$. 
              \end{theorem}

              \begin{definition}
              	Let $(V,\{\mathcal{C}_\alpha$ : $\alpha \in \Gamma\},e)$ be an Archimedean local ordered unit space. We define $LOMIN(V)$ to be the local operator system $\left(V,\left\{\{(\mathcal{C}_\alpha^n)^{\min}(V)\}_{n=1}^\infty\; : \; \alpha \in \Gamma\right\},e\right)$.
              	              \end{definition}
              
              Thus upto complete order isomorphism, LOMIN(V) can be identified with a subspace of $C(\cup S_\alpha(V_\alpha))$. Next we prove a universal property of this local operator system structure whose proof is now trivial.
              \begin{theorem}\label{univlomin}
              	Let $(V,\{V_\alpha^+\}$ : $\alpha \in \Gamma,e)$ be an Archimedean local ordered unit space. If $\left(W,\left\{\{\mathcal{D}_\beta^n\}_{n=1}^\infty\; : \; \beta \in \Omega \right\}\right)$ is a local matrix $\ast$-order vector space and $\phi$ : W $\to$ LOMIN(V) is a local positive linear map, then $\phi$ is completely local positive.
              	
              	Moreover, if $V'$=($V,\{\mathcal{C}_\alpha^n\};\alpha \in \Gamma,e$) is a local operator system with $C^1_\alpha$=$V_\alpha^+$ and such that, for every local operator system W, any positive map $\psi$ : W $\to$ $V'$ is completely local positive map, then the identity map is a unital complete local order isomorphism between $V'$ and LOMIN(V). 
              \end{theorem}
              	

              \begin{corollary}
              	Let $(V,\{V_\alpha^+\}$ :$\alpha \in \Gamma,e)$ be an Archimedean local ordered unit space. If $\left(V,\left\{\{\mathcal{C}_\alpha^n\}_{n=1}^\infty \; : \; \alpha \in \Gamma\right\},e\right)$ is any local operator system on $V$ with $\mathcal{C}_\alpha ^1$=$V_\alpha^+$, for each $\alpha$, then $\mathcal{C}_\alpha^n$ $\subseteq(\mathcal{C}_\alpha^n)^{\min}(V)$ for all $n$ and all $\alpha$. 	
              \end{corollary}
              \begin{proof} 
              	The identity map from $\left(V,\left\{\{\mathcal{C}_\alpha^n\}_{n=1}^\infty\; : \; \alpha \in \Gamma\right\},e\right)$ to LOMIN(V) is positive and hence by Theorem  \ref{univlomin} we have the result.
              \end{proof}
              
\subsection{LOMAX structure on Archimedean local ordered unit space}
              Now we discuss another important cone structures on an Archimedean local order unit space. For any $\ast$-vector space $V$, we identify the vector space $M_n(V)$ of $n\times n$ matrices with entries in $V$ with the (algebraic) tensor product $M_n\otimes V$ in a natural way.
              
              \begin{definition}\label{lomax1}
              	Let $(V,\{V_\alpha^+$ : $\alpha \in \Gamma\},e)$ be a local order $\ast$-vector space. Define $(\mathcal{D}_\alpha^n)^{\max}(V)$=\{$\sum\limits_{i=1}^k a_i\otimes v_i$; $v_i \in V_\alpha^+, a_i \in M_n^+, i=1,2,...,k;k \in \mathbb{N}$\} and $D^{\max}_\alpha(V)$=$\{(\mathcal{D}_\alpha^n)^{\max}(V)\}_{n=1}^\infty$ for each $\alpha$.
              \end{definition}
              
              \begin{lemma}\label{lomaxcone}
              	Let $(V,\{V_\alpha^+$ : $\alpha \in \Gamma\})$ be a local ordered $\ast$-vector space. Suppose that $P_\alpha^n\subseteq M_n(V)_h$ is a cone for each $n\in \mathbb{N}$ for each $\alpha$, the family \{$P_\alpha ^n$\} is a compatible matrix ordering i.e. $X^*P_\alpha^m X \subseteq P_\alpha^n$ where $x \in M_{m \times n}(\mathbb{C})$ and $P_\alpha$=$V_\alpha$ for each $\alpha$. Then $(\mathcal{D}_\alpha^n)^{\max}(V) \subseteq P_\alpha ^n$ $\forall n \in \mathbb{N}, \forall \alpha \in \Gamma$. 
              \end{lemma}
              \begin{proof} 
              	We have \{$P_\alpha ^n$\} family of cones at each matrix level with $P_\alpha^1$=$V_\alpha,$ $\forall \alpha$. If $X \in M_{n,1}$ then $XV_\alpha X^*$=$XP_\alpha^1 X^* \subseteq P_\alpha ^n$. It follows that $a \otimes v \in P_\alpha ^n$ for each $a \in M_n^+$ of rank one and $v \in V_\alpha ^+$. Since every element of $M_n^+$ is the sum of rank one elements of $M_n^+$, we conclude that $a\otimes v \in P_\alpha ^n,$ for all $a \in M_n^+$ and all $v \in V_\alpha ^+$. Thus $(\mathcal{D}_\alpha^n)^{\max}(V) \subseteq P_\alpha ^n,$ $\forall n \in \mathbb{N}$ and $\forall \alpha \in \Gamma$.
\end{proof}

              \begin{proposition}
              	Let $(V,\{V_\alpha^+$ : $\alpha \in \Gamma\},e)$ be an Archimedean local order unit space, then the cones $(\mathcal{D}_\alpha^n)^{\max}(V)$ are given by $$(\mathcal{D}_\alpha^n)^{\max}(V)=\{\gamma \mathrm{diag}(v_1,v_2,\ldots,v_m)\gamma ^*\; :\; \gamma \in M_{n,m},v_i \in V_\alpha^+,i=1,2,\ldots,m;m\in \mathbb{N}\}.$$
              \end{proposition}	
\begin{proof}               	Let $\mathcal{D}_\alpha ^n$ denote the set in the right hand side. By some simple computation we can show that $\mathcal{D}_\alpha ^n$ is a cone in $M_n(V)_h$.
               $\{\mathcal{D}_\alpha ^n\}_{n=1}^\infty$ is compatible for each $\alpha$. It is also clear that $\mathcal{D}_\alpha ^1$=$V_\alpha^+$. By Proposition \ref{closure}, we have that $(\mathcal{D}_\alpha^n)^{\max}(V) \subseteq \mathcal{D}_\alpha^n$ for each $n \in \mathbb{N}$ and $\alpha \in \Gamma$.	
              By the compatibility of $\{\mathcal{D}_\alpha ^n\}_{n=1}^\infty$ we have the reverse inclusion. 
              so we have $\mathcal{D}_\alpha ^n$=$(\mathcal{D}_\alpha ^n)^{\max}(V)$.
              \end{proof}
             
              The next result is now obvious.
              \begin{proposition}
              	Let $(V,\{V_\alpha^+$ : $\alpha \in \Gamma\},e)$ be an Archimedean local order unit space then $(\mathcal{D}_\alpha^n)^{\max}(V)$ defined in the Definition \ref{lomax1},  is local matrix ordering on V and $e$ is a matrix local order unit for this ordering.
              \end{proposition}
          
        \begin{remark}Let $(V,\{V_\alpha^+$ : $\alpha \in \Gamma\},e)$ be an Archimedean local order unit space.
        	By Lemma \ref{lomaxcone}, we have that $(\mathcal{D}_\alpha^n)^{\max}(V)$ $\subseteq$ $P_\alpha ^n$, where $\left\{\{P_\alpha^n\}_{n=1}^\infty : \alpha \in \Gamma\right\}$ is a local matrix ordering on V with $P_\alpha ^1$ = $V_\alpha ^+$ for each $\alpha \in \Gamma$. Thus $(\mathcal{D}_\alpha^n)^{\max}(V)$ is strongest cone structure on V. Though $e$ need not be Archimedean matrix order unit. As in Example \ref{3.6}(ii) if we take cone corresponding to the interval $[0,1]$ then I the constant function taking value 1 is not Archimedean matrix order unit explained in \cite{8}. But we can Archimedeanize it to get a local operator system on V, which is the strongest local operator system and we call that local opeartor system $LOMAX(V).$
      \end{remark}

\section{Tensor products of local operator systems}\label{s:4}
Inspired by the work \cite{kavruk2011tensor}, we here attempt to define tensor product in the category of local operator systems. Analogous to the minimal $\min$ and maximal $\max$ operator system tensor products, we also introduce $\lmin$ and $\lmax$ tensor products, and the notion of nuclearity in the category of local operator systems. Since every operator system is also a local operator system, this section is essentially generalization of \cite{kavruk2011tensor}. 
\begin{definition}\label{def_tensor}
	Given local operator systems $\left(V, \left\{\{\mathcal{C}_\alpha ^n\}_{n=1} ^{\infty}\; : \; \alpha \in \Gamma\right\}, e_V\right)$ and         \linebreak        
	$\left(W, \left\{\{\mathcal{D}_\beta ^n\}_{n=1} ^{\infty}\; : \; \beta \in \Lambda \right\},e_W\right),$ a local operator system structure ${l\tau}$ on $V \otimes W$ is a matricial cone structure given by $\left\{\{\mathcal{T}_{\gamma}^n\}_{n=1}^\infty :\gamma\in \Omega \right\}$  where $\Omega \cong \Gamma \times \Lambda$ such that:
	\begin{enumerate}
		\item $\left(V \otimes W,\left\{\{\mathcal{T}_{\gamma}^n\}_{n=1}^\infty :\gamma \in \Omega\}\right\},e_V \otimes e_W \right)$ is a local operator system.
		\item For every $\alpha \in \Gamma$ and $\beta \in \Lambda$, there exists a $\gamma \in\Omega$ such that $\mathcal{C}_\alpha^n \otimes \mathcal{D}_\beta^m \subseteq \mathcal{T}_{\gamma}^{nm}$ for all $n,m \in \mathbb{N}$ and for every $\gamma \in \Omega $, there exist $\alpha\in \Gamma$ and $\beta \in \Lambda$ such that $\mathcal{C}_\alpha^n \otimes \mathcal{D}_\beta^m \subseteq \mathcal{T}_{\gamma}^{nm}$ for all $n,m \in \mathbb{N}$. 
		\item If $\phi \in LUCP(V,M_n)$ and $\psi \in LUCP(W,M_m)$ w.r.t $\mathcal{C}_\alpha$ and $\mathcal{D}_\beta$ respectively, then $\phi \otimes \psi\in LUCP(V\otimes W,M_{nm})$ w.r.t $\mathcal{T}_{(\alpha,\beta)}$ for all $n,m \in \mathbb{ N}$. 
	\end{enumerate} 
	If for every pair of local operator systems V and W, $l\tau$ is a local operator system structure on $V \otimes W$, it is said to be a local operator system tensor product, denoted by $V \otimes _{l\tau} W$.
\end{definition}

As in the category of operator systems, a local opeartor system tensor product $l\tau$ is \emph{functorial} if for any four local operator systems $V_1,V_2,W_1,W_2$ we have that $\phi\in LUCP(V_1,V_2)$ and $\psi \in LUCP(W_1,W_2)$ implies the linear map $\phi \otimes \psi : V_1\otimes W_1 \to V_2 \otimes W_2$ belongs to $LUCP(V_1\otimes_ {l\tau}W_1,V_2\otimes _{l\tau}W_2)$. Given $V,W$ local operator systems, $l\tau$ is \emph{a symmetric local operator system tensor product} if the map $\theta: v\otimes w \to w \otimes v$ extends to a local unital complete order isomorphism from $V \otimes_{l\tau} W $ onto $W \otimes_{l\tau} V $. If for any three local operator systems $U,V,W$ the natural isomorphism yields a local complete order isomorphism from $(U\otimes_{l\tau} V)\otimes_{l\tau} W$ onto $U\otimes_{l\tau}(V \otimes _{l\tau} W)$ then $l\tau$ is called \emph{associative local operator system tensor product}. Given local operator systems $V_1\subseteq V_2$ and $W_1 \subseteq W_2$, if the inclusion map $V_1\otimes_{l\tau} W_1 \subseteq V_2\otimes_{l\tau}W_2$ is a local complete order isomorphism onto its range then $l\tau$ is \emph{injective local operator system tensor product.}  

\begin{proposition}\label{opsys_tp} Let $V$ and $W$ be two local operator systems such that $V=\underset{\longleftarrow}{\lim}V_\alpha$ and $W=\underset{\longleftarrow}{\lim}W_\beta$. Then corresponding to any operator system tensor product $\eta$, we have a local tensor product $\eta_l$ such that $V\otimes_{\eta_l} W$=$\underset{\longleftarrow}{\lim}V_\alpha\otimes_\eta W_\beta$.\end{proposition}
\begin{proof}  From Section \ref{projective}, we know that the $V$ and $W$ are abstract local operator systems with the family of cones $\mathcal{C}_\alpha ^n$=$\{((v_\alpha)_\alpha)_{ij}\in M_n(V): (v_\alpha)_{ij}\in M_n(V_\alpha)^+\}$ and $\mathcal{D}_\beta ^n$=$\{((w_\beta)_\beta)_{ij}\in M_n(W): (w_\beta)_{ij}\in M_n(W_\beta)^+\}$, respectively. Since, $\lim V_\alpha\otimes_\eta W_\beta$ is an operator system for each $\alpha$ and $\beta$, $\underset{\longleftarrow}{\lim}V_\alpha\otimes_\eta W_\beta$ being the projective limit of operator systems is a local operator system with family of cones defined as $\mathcal{T}_{(\alpha,\beta)}^{n(\eta_l)}=\{((v_\alpha \otimes w_\beta)_{(\alpha,\beta)})_{ij}\in M_n(V\otimes W)$: $((v_\alpha\otimes w_\beta)_{ij}\in M_n(V_\alpha \otimes _\eta W_\beta)^+\}$, that is, $(V \otimes _{\eta_l} W, \mathcal{T}_{(\alpha,\beta)}^{n(\eta_l}),e_V\otimes e_W)$ is a local operator system. Also, $\mathcal{C}_\alpha^n \otimes \mathcal{D}_\beta^m \subseteq  \mathcal{T}_{(\alpha,\beta)}^{nm(\eta_l)}$ for all $(\alpha,\beta)$ and $m,n \in \mathbb{N}$. Further note that if $\phi :V \to M_n$ and $\psi:W\to M_m$ are local unital completely positive maps w.r.t $\mathcal{C}_\alpha$ and $\mathcal{D}_\beta$, then there exist unital completely positive map $\phi_\alpha:V_\alpha \to M_m$ s.t. $\phi_\alpha \circ \pi_\alpha=\pi_m \circ \phi$ and $\psi_\beta \to M_n$ s.t. $\psi_\beta \circ \pi_\beta=\pi_n \circ \psi$. As $\eta$ is a operator system tensor product, so we have $\phi_\alpha\otimes \psi_\beta$: $V_\alpha \otimes _\eta W_\beta \to M_{mn}$ is unital completely positive map. So we have $\phi\otimes \psi$ : $V \otimes W \to M_{nm}$ is local unital completely positive w.r.t $\mathcal{T}_{(\alpha,\beta)}^{(\eta_l})$. Thus $\eta_l$ is a local operator system tensor product in the sense of Definition \ref{def_tensor} such that $V\otimes_{\eta_l} W$=$\underset{\longleftarrow}{\lim}V_\alpha\otimes_\eta W_\beta$.  
\end{proof}

Now we will define tensor product of local operator spaces with similar uniform conditions as that of tensor product of operator spaces. For complete details on tensor product of local operator spaces reader may refer \cite[Section 3.3]{webster1997local}.
\begin{definition}\label{tensor}
	Let $(E,||\cdot||_\rho)$ where $\rho \in \Gamma$ and $(F,||\cdot||_\sigma)$ where $\sigma \in \Lambda$ be two local operator spaces, we say $(E \otimes F,||.||_\delta)$ where $\delta \in \Omega$ with $\Omega\cong \Gamma \times \Lambda$ is local operator spaces tensor product $\tau$ if following conditions are satisfied
	\begin{enumerate}
		\item $(E \otimes F,||\cdot ||_\delta)$ is a local operator space.
		\item For any $\rho \in \Gamma$, $\sigma \in \Lambda$ there exists $\delta \in \Omega$ s.t. $||e\otimes f||_\delta^{nm}\leq ||e||_\rho^n||f||_\sigma^m$ and for any $\delta \in \Omega$ there exist $\rho \in \Gamma$ and $\sigma \in \Lambda$ s.t.  $||e\otimes f||_\delta^{nm}\leq ||e||_\rho^n||f||_\sigma^m$
		\item If $\phi:S \to M_n$ and $\psi:T\to M_m$ are local completely bounded maps w.r.t $\rho $ and $\sigma$ resp. then $\phi\otimes \psi: E \otimes F \to M_{nm}$ is a local completely bounded w.r.t. $(\rho,\sigma)$ and $||\phi \otimes \psi||_{(\rho,\sigma)}^{l.c.b}\leq ||\phi||_\rho^{l.c.b}||\psi||_\sigma^{l.c.b}.$
	\end{enumerate}
	
\end{definition}
Recall from Section \Ref{s:1} that every local operator system is also a local operator space whose matrix seminorms are determined by local matrix ordering. Thus it is important to understand the relationship between local operator system tensor products and local operator space tensor products. Let us first look at some elementary facts which are useful.
\begin{lemma}
	Let $P$, $A$ be operators on a quantized domain $D$ obtained by union of Hilbert spaces $H_\alpha$ with $P \in C_\alpha$. If $\begin{pmatrix}
	P & A\\
	A^* &P
	\end{pmatrix}\in C_\alpha^2$ then $A^*A\leq _\alpha ||P||_\alpha P$. In particular $||A||_\alpha \leq ||P||_\alpha$.
\end{lemma}
\begin{proof} This result can be proven easily by taking two cases $P|_{H_\alpha}=0$ and $P|_{H_\alpha}\neq 0$.
\end{proof}

\begin{theorem}\label{4}
	Let $S\subseteq \mathcal{A}$ be local operator system, and let $\mathcal{B}$ be a Pro-C* algebra, and let $\phi:S \to \mathcal{B}$ be local completely positive say w.r.t.($\alpha,\beta$) Then $\phi$ is local completely bounded and $||\phi(1)||_\beta=||\phi||_{(\alpha,\beta)}=||\phi||_{(\alpha,\beta)}^{l.c.b}$
\end{theorem}
\begin{proof} 	We have that $||\phi(1)||_\beta \leq ||\phi||_{(\alpha,\beta)}\leq ||\phi||_{\alpha,\beta}^{l.c.b}$, so it is enough to prove $||\phi||_{(\alpha,\beta)}\leq ||\phi(1)||_\beta$.
	Let $A=(a_{ij})\in M_n(S)$ with $||A||_\alpha\leq 1$. As $\begin{pmatrix}
	I_n &A \\
	A^* &I_n
	\end{pmatrix}$ is local positive w.r.t. $\alpha$ cone, therefore 
	$\phi_{2n}\bigg(\begin{pmatrix}
	I_n &A\\
	A^*  & I_n
	\end{pmatrix}\bigg)=\begin{pmatrix}
	\phi_n(I_n) & \phi_n(A)\\
	\phi_n(A)^* & \phi_n(I_n)
	\end{pmatrix}$ is positive w.r.t $\beta$ cone. And by above lemma we are done.
	\end{proof}

	\begin{proposition}\label{6}
	Let S be a local operator system and $f:S \to \mathbb{C}$ be a local positive map then $f$ is local completely positive.
\end{proposition}
\begin{proof} 
	As $f$ is local positive so there exists $\alpha$ s.t. $f(C_\alpha) \geq 0$. Let $(a_{i,j})\in C_\alpha ^n$ and $x=(x_1,x_2,..,x_n)\in \mathbb{C}^n$.
	We have that $<f((a_{i,j}))x,x>=f(\sum_{i,j}a_{i,j}x_j\bar{x_i})$ and the summation that $f$ is being evaluated at is the $(1,1)$ entry of element in $C_\alpha^n$ if $(a_{i,j}) \in C_\alpha^n$. Thus $f$ is local completely positive. 
	\end{proof}

\begin{theorem}\label{krein}
	Let $S$ be a local operator system contained in a Pro C*-algebra $A$ and let $\phi:S \to \mathbb{C}$ be local positive map then $\phi$ can be extended to local positive map on $A$. 
\end{theorem}
\begin{proof} 
	As we know from Proposition \Ref{6} that every local positive linear functional is local completely positive say w.r.t. $\alpha$, then by Theorem \Ref{4} $||\phi||_\alpha=|\phi(1)|=\phi(1)$ . By Hahn-Banach theorem $\phi$ has an extension $\tilde{\Phi}:A \to \mathbb{C}$ such that $||\tilde{\Phi}||_\alpha=\phi(1)=\tilde{\Phi}(1)$. By Corollary 4.1 in \cite{5} we have $\tilde{\Phi}$ has the required property.
	\end{proof}

Let $V$ be a local operator space and if $A \in M_n$ then $A_{(i,j)}$ denote the $(i,j) $th entry of $A$. If $\phi:V \to M_n$ is a linear map then we can associate a linear functional $s_\phi$ on $M_n(V)$ to $\phi$ such that $s_\phi((a_{i,j}))=\frac{1}{n}\sum_{i,j}\phi(a_{i,j})_{(i,j)}$. If $V$ contains the unit and $\phi(1)=1$ then $s_\phi(1)=1$. Also if $s:M_n(V)\to \mathbb{C}$, then we define $\phi_s:V \to M_n$ by $(\phi_s(a)_{(i,j)})=n.s(a \otimes E_{i,j})$ where $E_{i,j}$ is canonical marix with all entries zero except one at $(i,j)$-th place and $a \otimes E_{i,j}$ is in $M_n(V)$ with all entries zero except $a$ at the $(i,j)$ th place.
\begin{theorem}\label{3.2}
	Let $A$ be a Pro-$C^*$algebra with unit 1, let $S$ be a local operator system in $A$ and let $\phi: S \to M_n.$ The following are equivalent:
	\begin{enumerate}
		\item $\phi$ is local completely positive,
		\item $\phi$ is $n$-local positive,
		\item $s_\phi$ is local positive.
	\end{enumerate} 
\end{theorem}
\begin{proof}  Proof follows on the lines of \cite[Theorem 6.1]{paulsen}.
\end{proof}

\begin{theorem}\label{1.3}
	Let $A$ be a Pro $C^*$-algebra with unit and $S$ be a local operator system contained in $A$, and $\phi:S \to M_n$ local completely positive map. Then there exists a local completely positive map $\psi:A \to M_n$ which extends $\phi$.\
\end{theorem}
\begin{proof} 
	Let $s_\phi$ be the local positive linear functional on $M_n(S)$ associated with $\phi$ and let $s$ be the local positive linear functional on $M_n(A)$ which extends $s_\phi$ by Theorem\Ref{krein}. By Theorem \Ref{3.2}, we have that the map $\psi$ associated with $s$ is local completely positive. Clearly $\psi$ is the extension of $\phi$ as $s$ is the extenion of $s_\phi$.
\end{proof}

\begin{lemma}\label{1}
	Let $A$ and $B$ be Pro-$C^*$algebras with unit 1, let $V$ be a local operator space in $A$ and let $\phi:V \to B$. Define a local operator system $S^l_V\subseteq M_2(A)$ by $S^l_V=\bigg\{\begin{pmatrix} \lambda 1 & v  \\
	w^* &\mu 1\end{pmatrix}: \lambda,\mu\in \mathbb{C},v,w\in V\bigg\}$ and $\Phi:S^l_V\to M_2(B)$ via $\Phi\begin{pmatrix} \lambda 1 & v \\ w^* & \mu 1\end{pmatrix}=\begin{pmatrix}
	\lambda 1& \phi(v)\\
	\phi(w)^* & \mu 1
	\end{pmatrix}.$ If $\phi $ is local completely contractive then $\Phi$ is local completely positive.
\end{lemma}
\begin{proof} 
	As $\phi$ is local completely contractive there exist $(\alpha,\beta)$ such that $\phi_{(\alpha,\beta)}:V_\alpha \to B_\beta$ is completely contractive where $V_\alpha$ is operator space and $B_\beta$ is C*-algebra then by lemma 8.1 in \cite{paulsen} we have that $\Phi_{(\alpha,\beta)}: S_{V_\alpha} \to M_2(B_\beta)$ is completely positive so we have $\phi$ is local completely positive map.
	\end{proof}

 \begin{theorem}\label{1.6}
	Let $A$ be a Pro $C^*$-algebra with unit and $\phi:A \to C_\mathcal{E}^*(D)$ be local completely bounded map. Then there exist local completely positive maps $\phi_i:A \to  C_\mathcal{E}^*(D)$ such that the map $\Phi:M_2(A)\to C_{\mathcal{E}^2}^*(D^2)$ given by $\Phi\begin{pmatrix}
	a &b\\c&d
	\end{pmatrix}=\begin{pmatrix} \phi_1(a) & \phi(b)\\
	\phi^*(c) & \phi_2(d)\end{pmatrix}$ is local completely positive.
\end{theorem}
\begin{proof}  Using the Lemma \ref{1}, Theorem \ref{1.3} and \cite[Theorem 8.3]{paulsen} proof follows easily.
\end{proof}

\begin{proposition}\label{1.8}
	Let $S$ and $T$ be local operator systems and let $\phi_{i,j}:S \to T$, $1 \leq i,j \leq n$ be linear maps. Define $\Phi: S \to M_n(T)$ by $\Phi(s)=(\phi_{i,j}(s))$ and $\tilde{\Phi}:M_n(S)\to M_n(T)$ by $\tilde{\Phi}(s_{i,j})=(\phi_{i,j}(s_{i,j}))$. If $\tilde{\Phi}$ is local completely positive then $\Phi$ is local completely positive.
\end{proposition}
\begin{proof} 
	Define map $\delta :S \to M_n(S)$ by $\delta(s)=(s_{ij})$ where $s_{ij}=s$ for all $1 \leq i,j\leq n$ is local completely positive and $\Phi(s)=\tilde{\Phi} \circ\delta(s)$ which implies that $\phi$ is local completely positive map.
\end{proof}

\begin{proposition}\label{1.9}
	Let $S$ and $T$ be local operator systems and let $\tau$ be a local operator system structure on $S \otimes T$. If $\phi:S \to M_n$ and $\psi: T \to M_m$ are local completely positive then $\phi\otimes \psi :S\otimes _\tau T \to M_{nm}$ is local completely positive.
\end{proposition}
\begin{proof} From \cite[Exercise 6.2]{paulsen}, there exist local unital completely positive maps $\phi_1:S \to M_n$ and $\psi:T \to M_m$ and positive matrices $P\in M_n$ and $Q\in M_m$ such that $\phi(x)=P\phi_1(x) P$ and $\psi(y)=Q \psi_1(y) Q $ which further implies that $\phi \otimes \psi(x \otimes y)=(P \otimes Q)( \phi_1 \otimes \psi_1(x \otimes y))(P \otimes Q)$. Now by the third  property of tensor product of local operator system we have the required result.
\end{proof}

\begin{proposition}\label{1.13}
	Let $S$ and $T$ be local operator systems and let $\tau$ be a local operator system structure on $S\otimes T$. Then the local operator space $S \otimes_\tau T $ is a local operator space tensor product of the local operator spaces $S$ and $T$; that is the conditions in Definition \ref{tensor}  are satisfied.
\end{proposition}
\begin{proof} 
	Clearly condition (1) in Definition \ref{tensor} is satisfied.
	Now we will prove condition (2) is also satisfied. Let $e$ and $f$ denote the local order unit of $S$ and $T$ resp. For condition (2) it will be enough to assume that $||s||_\alpha^n\leq 1$ and  $||t||_\beta^m \leq 1$ and show that $||s\otimes t||_{(\alpha,\beta)}^{nm}\leq 1 $. As $||s||_\alpha^n\leq 1$ so we have $P=\begin{pmatrix}
	e_n & s\\
	s* & e_n
	\end{pmatrix} \in C_\alpha^{2n}.$ Similarly as  $||t||_\beta^m \leq 1$ we have $Q=\begin{pmatrix}
	f_m & t\\
	t* & f_m
	\end{pmatrix} \in D_\beta^{2m}.$ Now by using property (2) of tensor product of local operator system we can easily have the condition (2).
	%
 Now we will prove the condition (3) in Definition \ref{tensor}. To prove this condition it will be enough to consider $||\phi||_\alpha^{l.c.b}\leq 1$, $||\psi||_\beta^{l.c.b}\leq 1$. We have that 
	$||\phi||_\alpha^{l.c.b}\leq 1$ by Theorem \ref{1.6} there exist a local completely positive map $\Phi:M_2(S)\to M_2(M_n)$ by $\Phi 
	\begin{pmatrix}
	s_{11} & s_{12}\\
	s_{21} & s_{22}
	\end{pmatrix}=
	\begin{pmatrix}
	\phi_{1,1}(s_{11})& \phi(s_{12})\\
	\phi(s_{21}^*)^* &\phi_{2,2}(s_{22})
	\end{pmatrix}$
	where $\phi_{1,1},\phi_{2,2}:S \to M_n$ are local unital completely positive maps. Also there exists a similar local completely positive map $\Psi:M_2(T) \to M_2(M_m)$ with analogues properties.
	Let $\Phi_0=\Phi o\delta:S \to M_2(M_n)$ so that $\Phi_0(s)=\begin{pmatrix}
	\phi_{1,1}(s_{11})& \phi(s_{12})\\
	\phi(s_{21}^*)^* &\phi_{2,2}(s_{22})
	\end{pmatrix}$ and  $\Psi_0:T \to M_2(M_m)$ be defined in a similar way. By Proposition \ref{1.8} We have that $\Phi_0$ and $\Psi_0$ are local completely positive maps. By Proposition \ref{1.9} $\Phi_0\otimes \Psi_0: S\otimes _\tau T \to M_4(M_{mn})$ is local completely positive map and required result is now obvious.
	\end{proof}

We know that every operator space can be embedded into a operator system completely isometrically \cite[Section 3]{kavruk2011tensor}. We have similar result in case of local operator spaces.
If $V\subseteq C_\mathcal{E}(D)$ where $D=\cup H_\alpha$ then $S^l_V \subseteq C_{\mathcal{E}^2}^*(D^2)$ where $D^2=\cup H_\alpha \oplus H_\alpha$ is the local operator system given by $$S^l_V=\bigg\{\begin{pmatrix} \lambda I & v\\
w^* & \mu I\end{pmatrix}: \lambda,\mu \in \mathbb{C};v,w \in V\bigg\}$$
We see $V\subseteq S^l_V$ via the inclusion $v \rightarrow \begin{pmatrix}
0&v\\
0 &0
\end{pmatrix}$
\begin{remark}
	If $V=\lim\limits_{\leftarrow}V_\alpha$ then $S^l_V \cong \lim\limits_{\leftarrow} S_{V_\alpha}$ in local complete isomporphism sense.
\end{remark}
\begin{definition}
	Let $X$ and $Y$ be local operator spaces, and $\tau$ be a local opeartor systems structure on $S_X^l \otimes S_Y^l$. Then the embedding $X\otimes Y\subseteq S_X^l\otimes_\tau S _Y^l$ endows $X \otimes Y$ with a local opeartor space structure; we call the resulting local operator space the induced local operator space tensor product of $X$ and $Y$ and denote it by $X \otimes^{\tau} Y$ .
\end{definition}
\begin{proposition}
	let $X$ and $Y$ be local operator spaces and $\tau$ be a local operator system structure on $S_X^l\otimes S_Y^l$ and let $X\otimes^{\tau_l}Y$ be the induced local operator space tensor product. Then $X\otimes^{\tau}Y$ is a local operator space tensor product in the sense of Definition \ref{tensor}.
\end{proposition}
\begin{proof} 
	Clearly first condition is satisfied. Second condition is also satisfied from Proposition \ref{1.13} and the fact that the inclusion $X\subseteq S_X^l$ and $Y \subseteq S_Y^l$ are local complete isometry.
	Now we will see third condition can be proven easily using Lemma \ref{1} and Property (3) of tensor product of local operator system.
$	$ 
\end{proof}

\subsection{The minimal tensor product}
In this subsection, we give a construction of the local operator system tensor product $\mathrm{lmin}$, which is minimal among all local operator system tensor products.

Let $\left(V, \left\{\{\mathcal{C}_\alpha ^n\}_{n=1} ^{\infty}\; : \; \alpha \in \Gamma \right\}, e_V\right)$ and $\left(W, \left\{\{\mathcal{D}_\beta ^n\}_{n=1} ^{\infty}\; : \; \beta \in \Lambda \right\},e_W\right)$ be local operator systems. For each $\alpha \in \Gamma ,\beta \in \Lambda, n\in \mathbb{N}$, define
\begin{align*}
\begin{split}
\mathcal{T}_{(\alpha,\beta)}^{n(\lmin)}   := &\big\{(p_{ij}) \in M_n(V\otimes W) : ((\phi \otimes \psi)(p_{i,j}))\in M_{nkm}^+, \text{ for all } \phi:V \to M_k \\
& \text{ local unital completely positive map w.r.t. cone } \mathcal{C}_\alpha, \psi: W \to M_m \\ & \text{ local unital completely positive map w.r.t. cone } \mathcal{D}_\beta \text{ for all } k,m\in \mathbb{ N} \big\}
\end{split}
\end{align*}
\begin{lemma}\label{lmin_lucp}
	Let $\left(V, \left\{\{\mathcal{C}_\alpha ^n\}_{n=1} ^{\infty}\; : \; \alpha \in \Gamma\right\}, e_V\right)$ be a local operator system and $A \in M_n(V)$. If $\phi^{(n)}(A)\in M_{nk}^+$ for every $k \in \mathbb{N}$ and for every $\phi \in S^\alpha_k(V) $ where $S^\alpha_k(V) $ denotes the set of local unital completely positive maps $\Phi:V \to M_k
	$ w.r.t. cone $\mathcal{C}_\alpha,$ then $A\in \mathcal{C}_\alpha^n$.
\end{lemma}
\begin{proof} 
	By Theorem \ref{representation}, we may assume that $V \subseteq C_{\mathcal{E}}^*(D)$ where $D=\cup H_\alpha$. Suppose that $A=(a_{ij}) \in M_n(V)$ and $\phi^{(n)}(A)\in M_{nk}^+$ for every $\phi \in S^\alpha_k (V)$ and $k\in \mathbb{N}$. Let $\xi=(\xi_1,\xi_2,...,\xi_n)\in H_\alpha^n$ and $\phi:V \to M_n$ be the mapping given by $\phi(x)=\langle x \xi_j,\xi_i\rangle_{i,j}$. We can easily show that $\phi$ is local completely positive map w.r.t $\mathcal{C}_\alpha$
 hence $\phi^{(n)}(A)=(\phi(a_{ij}))_{i,j} \in M_{n^2}$ which further implies that
 $A|_{H_\alpha} \in M_n(B(H_\alpha))^+$. 
 \end{proof}


\begin{lemma}\label{lmin_lucp2}
	Let $\left(V, \left\{\{\mathcal{C}_\alpha ^n\}_{n=1} ^{\infty}\; : \; \alpha \in \Gamma\right\}, e_V\right)$ and $\left(W, \left\{\{\mathcal{D}_\beta ^n\}_{n=1} ^{\infty}\; : \; \beta \in \Lambda \right\},e_W\right)$ be local operator systems and $A \in M_n(V)\otimes W $. If $(\phi^{(n)}\otimes \psi)(A)\geq 0$ for all $\phi \in S^\alpha_\infty(V)$ and all $\psi \in S^\beta_\infty(T)$, then $(\Phi \otimes \psi )(A) \geq 0$ for all $\Phi \in S^\alpha_\infty(M_n(V))$ and all $\psi \in S^\beta_\infty (W)$ where $S_\infty^\alpha(V)=\cup S_k^\alpha(V)$.	
\end{lemma}
\begin{proof}  This lemma can be easily proven using Lemma \ref{lmin_lucp} and \cite[Lemma 4.2]{kavruk2011tensor}.
\end{proof}

\begin{lemma}\label{lmin_lucp3} If $\phi \in S_\infty^\alpha(V)$ and $\psi\in S_\infty^\beta (W)$ then $(\phi \otimes \psi)^{(n)}=\phi^{(n)}\otimes \psi$.
\end{lemma}
\begin{proof} 	It is sufficient to check the equality on elementary tensors of the form $A=X\otimes y$, where $X=(x_{ij}) \in M_n(V)$ and $y \in W$. For such a $A$, we have that $(\phi^{(n)}\otimes \psi)(A)=(\phi(x_{ij}))_{ij}\otimes \psi(y)$. On the other hand, $(\phi \otimes \psi)^{(n)}(A)=((\phi \otimes \psi)(x_{ij}\otimes y))_{ij}=(\phi(x_{ij})\otimes \psi(y))_{ij}$.
\end{proof}

\begin{theorem}\label{lmin_tp}
	Let $\left(V, \left\{\{\mathcal{C}_\alpha ^n\}_{n=1} ^{\infty}\; : \; \alpha \in \Gamma\right\}, e_V\right)$ and $\left(W, \left\{\{\mathcal{D}_\beta ^n\}_{n=1} ^{\infty}\; : \;\beta \in \Lambda \right\},e_W\right)$ be local operator systems and let $i_V:V \to C_{\mathcal{E}}^*(D)$ and $i_W :W \to C_{\mathcal{G}}^* (F)$ be embedding that are local unital complete order isomorphism onto their ranges. The family $\left\{\{\mathcal{T}_{(\alpha,\beta)}^{n(\lmin)}\}_{n=1}^\infty \; :\; (\alpha,\beta)\in \Gamma \times \Lambda \right\}$ is the local operator system structure on $V\otimes W$ arising from the embedding $i_V \otimes i_W: V\otimes W \to C_{\mathcal{E\times G}}^*(D\otimes F)$ and we have  $D\otimes F \cong \cup_{(\alpha,\beta) \in \Gamma\times \Lambda} H_\alpha \otimes K_\beta $ where $D=\cup H_\alpha$  and $F=\cup K_\beta$.
\end{theorem}
\begin{proof} 	Let $R \in \mathcal{T}_{(\alpha,\beta)}^{n(\lmin)}(V,W)$. We show that $(i_V\otimes i_W)^{(n)}(R) \in  (C_{\mathcal{E\times G}}^*(D\otimes F)^n)_{(\alpha,\beta)}^+.$ Say $S=(i_V\otimes i_W)^{(n)}(R),$ we need to show that $S|_{(H_\alpha\otimes K_\beta)^n}\geq 0$. Suppose that $S=\sum \limits_{r=1}^l X_r\otimes y_r$, where $X_r\in M_n(i_V(V))$ and $y_r\in i_W(W)$ for $r=1,2,..,l$. Let $\xi_s \in H_\alpha ^n $ and $\eta_s \in K_\beta$ for $s=1,2,..,k.$ and set $\xi=\sum\limits _{s=1}^k \xi_s \otimes \eta_s.$ Let $\Phi: M_n(i_V(V)) \to M_k$, where $E_\alpha=\{T\in C_\mathcal{E}^*(D): T|_{H_\alpha}\geq 0\}$ cones in $i_V(V)$; given by $(\Phi((X_{ij})_{i,,j}))_{s,t}$=$\langle (X_{ij})\xi_t,\xi_s\rangle_{s,t}$ and let $\psi:i_W(W) \to M_k$ where $F_\beta =\{T \in C_\mathcal{F}^*(G): T|_{K_\beta} \geq 0\}$ cones in $i_W(W)$; given by $(\psi(y))_{s,t}=\langle y \eta_t,\eta_s\rangle_{s,t} $.
	As in the proof of Lemma \ref{lmin_lucp}, $\Phi$ and $\psi$ are local completely positive w.r.t. cones $E_\alpha$ and $F_\beta,$ respectively. Since $S \in \mathcal{T}_{(\alpha,\beta)}^{n(min)}(i_V(V),i_W(W))$ say $S=(s_{ij})$. By Lemma \ref{lmin_lucp3}, $(\phi_0^{(n)}\otimes \psi_0)(s_{ij})\in M_{nk^2}^+$ for all $\phi_0 :i_V(V) \to M_k$ local unital completely positive maps w.r.t. $E_\alpha$ and all $\psi_0 : i_W(W) \to M_k $ local unital completely positive maps w.r.t $F_\beta$. By Lemma \ref{lmin_lucp2}, $(\Phi\otimes\psi(S))\geq 0$ . 
	It follows that $S|_{{(H_\alpha\otimes K_\beta)}^n}\in (B(H_\alpha \otimes K_\beta)^n)^+$ and thus claim is proved. We now prove the converse part. Let $T_{(\alpha,\beta)}^n$ is a cone in $M_n(V\otimes W)$ arising from the inclusion $i_V(V) \otimes i_W(W)$ into $C_{(\mathcal{E \times G})}^*(D\otimes F)$. We show that $T_{(\alpha,\beta)}^n\subseteq \mathcal{T}_{(\alpha,\beta)}^{n(min)}$. Suppose that $\phi: V \to M_m$ and $\psi : W \to M_k$ are local unital completely positive maps w.r.t $\mathcal{C}_\alpha$, $\mathcal{D}_\beta$ respectively. As $\phi$ and $\psi$ are local unital completely positive maps so there exist $\phi_\alpha:V_\alpha \to M_m$ and $\psi_\beta: W_\beta \to M_n$ unital completely positive map. Clearly, $\phi_\alpha$ and  $\phi_\alpha$ are unital complete order isomorphisms onto their range. Now by using Arveson's extension theorem and $C^*\text{-}$ algebra theory we can easily have that $(\phi\otimes \psi)^{(n)}(p_{ij}) \in M_{nkm}^+$
	Hence we have $T_{(\alpha,\beta)}^n$=$\mathcal{T}_{(\alpha,\beta)}^{n(\lmin)}$. So $\mathcal{T}_{(\alpha,\beta)}^{n(\lmin)}$ is local operator system structure on $V\otimes W$ with Archimedean matrix order unit $e_V\otimes e_W$.
	\end{proof}

\begin{definition}
	We call $\left(V \otimes W,\left\{\{\mathcal{T}_{(\alpha,\beta)}^{n(\lmin)}\}_{n=1}^\infty :(\alpha,\beta)\in \Gamma \times \Lambda\right\},e_V \otimes e_W \right)$ \emph{the minimal local tensor product} of $V$ and $W$ and denote it by $V \otimes _{\lmin} W$.
\end{definition}
\begin{corollary}
	Let $\left(V, \left\{\{\mathcal{C}_\alpha ^n\}_{n=1} ^{\infty}\; : \; \alpha \in \Gamma\right\}, e_V\right)$ and $\left(W, \left\{\{\mathcal{D}_\beta ^n\}_{n=1} ^{\infty}\; : \;\beta \in \Lambda \right\},e_W\right)$ be local operator systems then $\mathcal{T}_{(\alpha,\beta)}^{n(\lmin)}$ is the minimal local operator system tensor product structure of V and W.
\end{corollary}
\begin{proof} 	Let  $T_{(\alpha,\beta)}^n$ be a local operator system tensor product on $V \otimes W$ and denote it by $V \otimes_{l\eta}W$. Let $\Phi:V \otimes_{l\eta}W \to V \otimes_{(\lmin)}W$ be the identity map. By using condition (3) in Definition \ref{def_tensor}, we have $T_{(\alpha,\beta)}^n\subseteq \mathcal{T}_{(\alpha,\beta)}^{n(\lmin)}$ for every $n\in \mathbb{N}$. Thus $\mathcal{T}_{(\alpha,\beta)}^{n(\lmin)}$ is the minimal structure.
\end{proof}

\begin{corollary}
	The mapping $\lmin$: $\mathcal{LO} \times \mathcal{LO} \to \mathcal{LO}$ sending $(V,W)$ to $V \otimes _{\lmin} W$ is an injective, associative, symmetric, functorial local operator system tensor product.
\end{corollary}


\begin{remark} $V\otimes _{\lmin} W \cong \underset{\longleftarrow}{\lim} V_\alpha \otimes _{\min} W_\beta$=$V\otimes_{\min_l} W;$ where $V=\underset{\longleftarrow}{\lim}V_\alpha$ and $W=\underset{\longleftarrow}{\lim}W_\beta$, and  $\min_l$ is the tensor product in the sense of Proposition \ref{opsys_tp}.  
	
\end{remark}
\	\begin{proposition}
	Let $V$ and $W$ be A.L.O.U spaces and tensor product $V\otimes W$ is equipped with the family of cones as $Q_{min}^{(\alpha,\beta)}=\{u\in V\otimes W: (f\otimes g)(u)\geq 0$ for all $f \in S^\alpha(V),g\in S^\beta(W) \}$. Then $LOMIN(V)\otimes_{lmin}LOMIN(W)=LOMIN(V\otimes W)$. 
\end{proposition} 
\begin{proof} 	We have $LOMIN(V)\subseteq C(X)$ where $X=\cup S^\alpha(V)$. Similarly we have $LOMIN(W)\subseteq C(Y)$ where $Y=\cup S^\beta(W)$. By injectivity of $lmin$, we have that $LOMIN(V)\otimes_{lmin}LOMIN(W)$ is a local operator subsystem of $C(X)\otimes_{lmin}C(Y)$. Denote the matrix ordering on $LOMIN(V \otimes W)$ by $\{Q_n^{(\alpha,\beta)}\}$ and matrix ordering on $LOMIN(V)\otimes_{lmin}LOMIN(W)$ by $\{D_n^{(\alpha,\beta)}\}$. Since $LOMIN(V\otimes W)$ is the minimal local operator system structure $(V \otimes W,Q_{min}^{(\alpha,\beta)})$, we have that $D_n^{(\alpha,\beta)}\subseteq Q_n^{(\alpha,\beta)}$ for all $n\in \mathbb{N}$. Now we will show that reverse inclusion also holds. Suppose that $X=(x_{ij})\in Q_n^{(\alpha,\beta)}$, we have that $\sum_{i,j=1}^{n}\bar{\lambda_i}\lambda_j x_{i,j}\in Q_{min}^{(\alpha,\beta)}$ for all $\lambda_1,\lambda_2,....,\lambda_n\in \mathbb{C}$. Let $\lambda=(\lambda_1,...,\lambda_n)^t$ and we have that for $f\in S^\alpha(V)$ and $g\in S^\beta(W)$ , $<(f\otimes g)(x_{i,j}))_{i,j} \lambda, \lambda>=\sum_{i,j=1}^n\bar{\lambda_i}\lambda_j(f\otimes g)(x_{ij})\geq 0$. It implies that $X \in D_n^{(\alpha,\beta)}$. 
\end{proof}

\subsection{The maximal tensor product}
We now construct the maximal local operator system tensor product and prove that this is in fact maximal tensor product structure possible in the category of local operator systems.

Let $\left(V, \left\{\{\mathcal{C}_\alpha ^n\}_{n=1} ^{\infty}\; : \;\alpha \in \Gamma\right\}, e_V\right)$ and $\left(W, \left\{\{\mathcal{D}_\beta ^n\}_{n=1} ^{\infty}\; : \; \beta \in \Lambda \right\},e_W\right)$ be two local operator systems. For each $n\in \mathbb{N}$ and $(\alpha, \beta)\in \Gamma\times\Lambda,$ define
$$\mathcal{K}_{(\alpha,\beta)}^{n(max)}:= \{\alpha (P\otimes Q) \alpha^*\; : \; P\in \mathcal{C}_\alpha ^k \text{ and } Q \in \mathcal{D}_\beta ^m, \alpha\in M_{n,km}, k,m\in \mathbb{N}\}.$$ 

\begin{lemma}\label{lmax_comp}
	For local operator systems $V$ and $W$ be local operator systems and $\{ T_{\gamma}^n: \gamma \in \Omega\}$ be a compatible collection of cones, where $T_{\gamma}^n\subseteq M_n(V\otimes W)$ satisfies (2) in Definition \ref{def_tensor}. Then for each $\gamma \in \Omega$ there exist $(\alpha,\beta) \in \Gamma \times \Lambda$ such that $\mathcal{K}_{(\alpha,\beta)}^{n(\mathrm{lmax})} \subseteq T_{\gamma}^n$ for each $n \in \mathbb{N}$. 
\end{lemma}
\begin{proof} 	By condition (2), we have for any $\gamma$ there exist $\alpha$, $\beta$ s.t. $P \otimes Q \in T_{\gamma}^{km}$ where $P\in \mathcal{C}_\alpha ^k$ and $Q \in \mathcal{D}_\beta ^m$ for any $k,m \in \mathbb{N}$.
	The compatibility of $\{T_{\gamma}^n\}$ implies that $\lambda(P\otimes Q) \lambda ^* \in T_{\gamma}^n$ for every $\lambda \in M_{n,km}$. Thus $\mathcal{K}_{(\alpha,\beta)}^{n(\mathrm{lmax}} \subseteq T_{\gamma}^n$.
	\end{proof}

\begin{proposition}
	For any two local operator systems $V$ and $W$, the family $\left\{\{K_{(\alpha,\beta)}^{n(\mathrm{lmax})}\}_{n=1}^\infty\; : \; (\alpha,\beta)\in \Gamma\times \Lambda \right\}$ is a matrix ordering on $V\otimes W$ with order unit $e_V\otimes e_W$. 
\end{proposition}
\begin{proof} 	Since for any $X\in M_{n,m}$, $X^*K_{(\alpha,\beta)}^{n(\mathrm{lmax})}X\subseteq K_{(\alpha,\beta)}^{m(\mathrm{lmax})}$ proof follows immediately.
\end{proof}

\begin{definition}
	Let $\{\mathcal{T}_{(\alpha,\beta)}^{n(\mathrm{lmax})}\}_{n=1}^\infty$ be the Archimedeanization of the matrix ordering $\{K_{(\alpha,\beta)}^{n(\mathrm{lmax})}\}_{n=1}^\infty$ for local operator systems $V$ and $W$, we call the 
	local operator system $\left(V \otimes W,\left\{\{\mathcal{T}_{(\alpha,\beta)}^{n(\mathrm{lmax})}\}_{n=1}^\infty :(\alpha,\beta)\in \Gamma \times \Lambda\right\},e_V \otimes e_W \right)$ the maximal local operator system tensor product of $V $ and $W$ and denote it by $V\otimes _{\mathrm{lmax}}W$.
\end{definition} 
\begin{proposition}
	Let $V$ and $W$ be local operator systems. Then $V\otimes _{\mathrm{lmax}}W$ is local operator system tensor product. Moreover, if $l\tau$ is a local operator system tensor product structure on $V\otimes W$, then $\mathrm{lmax}$ is larger than $l\tau$. 
\end{proposition}
\begin{proof} 	Let $V$ and $W$ be local operator system. By its definition, the family  $\{\mathcal{T}_{(\alpha,\beta)}^{n(\mathrm{lmax})}\}_{n=1}^\infty$ satisfies property (1) and (2) of Definition \ref{def_tensor}. Since $\mathcal{T}_{(\alpha,\beta)}^{n(\mathrm{lmax})}\subseteq \mathcal{T}_{(\alpha,\beta)}^{n(\lmin)}$ it follows that $V\otimes_{\mathrm{lmax}} W$ satisfies property (3). In fact, $\mathrm{lmax}$ is maximal local operator system tensor product follows from Lemma \ref{lmax_comp}.
\end{proof}

Following now follows immediately \cite[theorem 5.5]{kavruk2011tensor}.
\begin{theorem}
	The mapping $\mathrm{lmax}:\mathcal{LO} \times \mathcal{LO} \to \mathcal{LO}$ sending $(V,W)$ to $V\otimes _{\mathrm{lmax}} W$ is a symmetric, associative, functorial local operator system tensor product. 
\end{theorem}

\begin{remark}
	Let $V$ and $W$ be two local operator systems s.t. $V=\underset{\longleftarrow}{\lim}V_\alpha$ and $W=\underset{\longleftarrow}{\lim}W_\beta$. Then we have $V\otimes _{\mathrm{lmax}} W \cong \underset{\longleftarrow}{\lim} V_\alpha \otimes _{max} W_\beta$=$V\otimes_{max_l} W$.
	Define $\phi : V\otimes _{\mathrm{lmax}} W \to \underset{\longleftarrow}{\lim} V_\alpha \otimes _{max} W_\beta$ by $\phi((v_\alpha)\otimes (w_\beta))=(v_\alpha \otimes w_\beta)$. So we have here that $\mathcal{T}_{(\alpha,\beta)}^{n(max_l)}\subseteq \mathcal{T}_{(\alpha,\beta)}^{n(l max))}$ as let $p \in \mathcal{T}_{(\alpha,\beta)}^{n(max_l)} $ so we have $(p)_{(\alpha,\beta)}\in V_\alpha \otimes_{max} W_\beta$ and by definition of operator system max tensor product, there exist $\lambda \in M_{mk,n}$ s.t. $(p)_{(\alpha,\beta)} =\lambda^* (v_\alpha \otimes w_\beta)\lambda $ where $v_\alpha \in (V_\alpha)^+_{m}$ and $w_\beta \in (W_\beta)^+_{k}$ so we have $p=\lambda^* ((v_\alpha) \otimes (w_\beta))\lambda$ where $v_\alpha \in (V_\alpha)^+_{m}$ and $w_\beta \in (W_\beta)^+_{k}$. Hence, we have $p \in \mathcal{T}_{(\alpha,\beta)}^{n(l max))} $. On the other hand, we have $\mathcal{T}_{(\alpha,\beta)}^{n(\mathrm{lmax})}\subseteq \mathcal{T}_{(\alpha,\beta)}^{n(max_l))}$ by using maximality of $\mathcal{T}_{(\alpha,\beta)}^{n(\mathrm{lmax})}$.
\end{remark}
If $X$ and $Y$ are local operator spaces, then we let $X \otimes_{\wedge} Y$ denote the local operator space projective tensor product. We refer the reader to \cite[Section 3]{webster1997local} for more details of this tensor product.
\begin{theorem}
	Let $X$ and $Y$ be local operator spaces. Then $X\otimes^{lmax}Y$ coincides with the local operator space projective $X {\otimes}_{\wedge} Y$.
\end{theorem}
\begin{proof} This result can be done on the same lines \cite[Theorem 5.9]{kavruk2011tensor}
with seminorms.
\end{proof}
\begin{proposition}
	Let $(V,\{V_\alpha^+\})$ and $(W,\{W_\beta^+\})$ be A.L.O.U spaces. Equip the tensor product $V\otimes W$ with the Archimedeanization of the cone 
	$P^{(\alpha,\beta)}_{lmax}=\{\sum_{i=1}^{k}v_i\otimes w_i: v_i\in V_\alpha^+,w_i\in W_\beta^+ \text{ and }k \in \mathbb{N}\}$. Then $LOMAX(V) \otimes_{lmax} LOMAX(W)=LOMAX(V \otimes W)$.
\end{proposition}
\begin{proof} 	We have that local matrix ordering on $LOMAX(V)$ is the Archimeanization of $\{D_n^{lmax (\alpha)}(V)\}_{n=1}^\infty$ where $D_n^{lmax (\alpha)}(V)=\{\sum_{j=1}^{k}a_j \otimes v_j , a_j \in M_n^+,v_j\in V_\alpha^+,k\in \mathbb{N}\}$. Similarly we have $\{D_n^{lmax (\beta)}(W)\}_{n=1}^\infty$ with resp. to the cone $W_\beta^+$ and $\{D_n^{lmax (\alpha,\beta)}(V \otimes W)\}_{n=1}^\infty$ with respect to cone $P^{(\alpha,\beta)}_{lmax}$. For the required result we need to show that $D_n^{lmax (\alpha,\beta)}(V \otimes W)=\{A(P \otimes Q)A^*: P\in D_k^{lmax (\alpha)}(V),Q\in D_m^{lmax (\beta)}(W), A\in M_{n,km}\}$.
	Let $D_n^{(\alpha,\beta)}$ denotes the right hand side of the last equation. If $a_j\in M_n^+$ and $\sum_{i=1}^{k_j}v_i^j\otimes w_i^j \in P_{lmax}^{(\alpha,\beta)},j=1,..,l$ where $v_i^j\in V_\alpha^+$ and $w_i^J\in W_\beta^+$ then $\sum_{j=1}^{l}a_j\otimes \bigg(\sum_{i=1}^{k_j}v_i^j\otimes w_i^j\bigg)=\sum_{j,i}a_j\otimes v_i^j \otimes w_i^j$. Since $\sum_i a_j \otimes v_i^j \in D_n^{lmax(\alpha)}(V)$ for each $j$ we have that $\sum_{j,i}a_j\otimes v_i^j \otimes w_i^j\in D_n^{(\alpha,\beta)}(V)$. Thus we have that $D_n^{lmax(\alpha,\beta)}(V \otimes W) \subseteq D_n^{(\alpha,\beta)}$. Now we will prove the other containment, which is obvious using the compatibility of the family $\{D_n^{lmax(\alpha,\beta)}(V \otimes W)\}$. 
	\end{proof}

%

\begin{remark} One can also relate the notion of nuclearity to the category of local operator system. Let $\{V_\alpha: \alpha \in \Gamma\}$ be projective system of $(\eta,\gamma)$-nuclear operator systems then it follows from Proposition \ref{opsys_tp} that, $V=\underset{\longleftarrow}{\lim}V_\alpha$ is a $(\eta_l,\gamma_l)$-nuclear local operator system.	
\end{remark}



\section*{Acknowledgements}
\noindent Research of the first author is supported by the DST-INSPIRE Fellowship (Grant no. IF180802).

\bibliographystyle{plain}
\bibliography{LocalOperatorSystem}
\end{document}